\documentclass[12pt]{amsart}

\usepackage[top=30truemm,bottom=30truemm,left=25truemm,right=25truemm]{geometry}
\usepackage{mathrsfs}
\usepackage{enumitem}
\usepackage{amsmath, amsthm, amssymb}
\usepackage{color}
\usepackage{bm}
\usepackage{amsfonts,amssymb}
\usepackage{dsfont}
\usepackage{amscd}
\usepackage{extarrows}
\usepackage{amsmath}
\usepackage{mathrsfs}
\usepackage{amscd}
\usepackage[all]{xy}
\usepackage{geometry}
\usepackage[colorlinks]{hyperref}

\hypersetup{pdfencoding=auto}
\usepackage[capitalize]{cleveref}
\newtheorem{theorem}{Theorem}[section]
\newtheorem{definition}[theorem]{Definition}
\newtheorem{lemma}[theorem]{Lemma}
\newtheorem{corollary}[theorem]{Corollary}
\newtheorem{proposition}[theorem]{Proposition}
\newtheorem{remark}[theorem]{Remark}

\newcommand{\be}{\begin{equation}}
\newcommand{\ee}{\end{equation}}
\newcommand{\bea}{\begin{eqnarray}}
\newcommand{\eea}{\end{eqnarray}}
\newcommand{\ben}{\begin{eqnarray*}}
	\newcommand{\een}{\end{eqnarray*}}
\newcommand{\bt}{\begin{split}}
	\newcommand{\et}{\end{split}}
\newcommand{\bet}{\begin{equation}}

\begin{document}
\title[Mixed Hessian inequalities]
{Mixed Hessian inequalities on Hermitian manifolds and applications}

\author[H. Sun]{Haoyuan Sun}
\address{Haoyuan Sun: School of Mathematical Sciences\\ Beijing Normal University\\ Beijing 100875\\ P. R. China}
\email{202531130037@mail.bnu.edu.cn}

\begin{abstract}
     Let $(X,\omega)$ be a compact Hermitian manifold of complex dimension $n$. In this paper we establish a Ko\l odziej-Nguyen type weak convergence theorem of complex Hessian operators. Utilizing this result, we prove a general mixed Hessian inequality with respect to a background Hermitian metric, covering both local and global case. As an application, we prove the existence of bounded solutions of complex Hessian equations where the right-hand side measure is well dominated by capacities.
\end{abstract}


\maketitle

\tableofcontents

\section{Introduction}
Since Yau's landmark resolution of the Calabi conjecture, the complex Monge-Amp\`ere equation has become a central topic at the intersection of partial differential equations and complex geometry. A vast body of profound results has since been established; see, for instance, \cite{Yau78,  BT82, Kol98, TW10, EGZ09, BEGZ10, Ngu16, GL23, LWZ24} and references therein. 

Let $(X,\omega)$ be a compact Hermitian manifold of dimension $n$, and fix an integer $1\leq m<n$.  We are interested in the complex Hessian equation, a natural generalization of the complex Monge-Ampère equation, which takes the form:
$$
(\omega+dd^cu)^m\wedge\omega^{n-m}=fdV_X.
$$
This equation has been  a subject of extensive research over the past two decades. A primary motivation for its study stems from the well-developed theory of its real counterpart, see for example \cite{CNS85,CW01,TW97,TW99,TW02,Wang09}. Furthermore, various Hessian-type nonlinear equations, such as the J-equation and the deformed Hermitian Yang-Mills equation, have yielded profound applications in differential geometry and physics, such as \cite{CJY20, DP21, Song20, Chen21, Chen22}.

In \cite{Błocki05}, B\l ocki observed that the methods of pluripotential theory developed by Bedford and Taylor \cite{BT82} can be applied to define the $m$-Hessian measure of bounded $m$-subharmonic functions. This was later generalized to compact K\"ahler manifolds by \cite{Lu13}, \cite{Lu15} and subsequently to Hermitian manifolds by \cite{GN18},\cite{KN25b}. Smooth solutions of Complex Hessian equations on compact K\"ahler manifolds were obtained by \cite{DK17} and on Hermitian manifolds by \cite{Zhang15} and \cite{Szé18}. Correspondingly, theory of weak solutions has also been successfully developed by \cite{Lu15, KN16, LeV23, Sun24, GL25, CX25, KN25b, Fang25b, PSWZ25}, to name a few. In these proofs, the mixed Hessian inequalities appeared to be particularly useful in controlling the twist constants that appear on the right-hand side of the equation. The uniqueness of bounded solutions also relies heavily on this inequality, as noted in \cite[Proposition 3.16]{KN16}. 

Mixed Hessian inequalities with respect to a background K\"ahler metric have been established in \cite{DL15} using a clever regularization process. In their proofs, they solved a special class of complex Hessian equations (see \cref{Hessian equation for fHm} below) using the variational method developed by \cite{BBGZ13}. However, this approach cannot be directly extended to Hermitian manifolds due to the nonclosedness of the Hermitian form. In this article, we overcome this difficulty by proving a criterion for weak continuity of Complex Hessian measures analogous to the Monge-Amp\`ere case in \cite{KN25a}, as follows.

Let $\gamma\in\Gamma_m(\omega)$ be an $(\omega,m)$-positive form on $X$. The complex Hessian measure of a bounded potential $u\in\operatorname{SH}_m(X,\gamma,\omega)\cap L^\infty(X)$ is defined by 
$$
H_m(u):=(\gamma+dd^cu)^m\wedge\omega^{n-m}.
$$
For any Borel subset $E$ of $X$, the $m$-capacity of $E$ is defined as
$$
\operatorname{Cap}_{\gamma,\omega,m}(E):=\sup \left\{\int_{E} H_{m}(u)| u\in\operatorname{SH}_m(X, \gamma,\omega) \cap L^{\infty}(X), -1 \leq u\leq 0\right\}.
$$
A sequence of Borel functions $f_j$ is said to converge to $f$ in capacity on $X$ if for each fixed $\delta>0$, we have  
$$
\lim_{j\rightarrow +\infty} \operatorname{Cap}_{\gamma,\omega,m}\left( \left\{|f_j-f|>\delta \right \}\right)=0.
$$
In the first section, we will prove the following weak convergence Theorem of Hessian operators analogous to \cite[Theorem 1.1]{KN25a}:
\begin{theorem}\label{main weak convergence introduction}
    Let $\{u_j\}$ be a sequence of uniformly bounded $(\gamma,\omega,m)$-subharmonic functions. Assume that $H_{m}(u_j)\leq C_1H_{m}(\varphi_j)+C_2H_{m}(\psi_j)$ for some uniformly bounded sequence $\varphi_j\rightarrow\varphi \in\operatorname{SH}_m(X,\gamma,\omega)$ and $\psi_j\rightarrow\psi\in \operatorname{SH}_m(X,\gamma,\omega)$ in capacity. If $u_j\rightarrow u$ in $L^1(X)$, then a subsequence of $u_j$ converges in capacity to $u$. In particular, a subsequence of $H_{m}(u_j)$ converges weakly to $H_{m}(u)$.
\end{theorem}

\cref{main weak convergence introduction} can be used to solve a particular class of complex Hessian equations (see \cref{Hessian equation for fHm}), which is crucial in the proof of mixed Hessian inequalities. In \cref{section 4}, we adapt the approach of \cite{DL15} to establish the mixed Hessian inequalities with respect to a general Hermitian form. As in \cite{DL15}, we first work on compact complex manifolds, the local version is derived from the global version via an embedding and extension technique. A key feature of our strategy is a localization argument. Instead of the torus embedding used in the Kähler case (\cite{DL15}), we regard the local Euclidean ball as an open submanifold of the complex projective space $\mathbb{CP}^n$. The main result is as follows:
\begin{theorem}\label{local mix type introduction}
    Let $(B,\omega)$ be a small ball in $\mathbb{C}^n$ equipped with a Hermitian metric $\omega$ and let $\mu$ be a positive Radon measure on $B$, which is absolutely continuous with respect to the Hessian measure $H_m(\varphi):=(dd^c\varphi)^m\wedge\omega^{n-m}$ for some bounded $(\omega,m)$-subharmonic function $\varphi$ (see \cref{def: def of m-sh} below). Let $u_1,...,u_m$ be bounded $(\omega,m)$-subharmonic functions such that $(dd^cu_j)^m\wedge\omega^{n-m}\geq f_j\mu$ for some $0\leq f_j\in L^1(\mu)$ . Then
    $$
dd^cu_1\wedge...\wedge dd^cu_m\wedge\omega^{n-m}\geq(f_1...f_m)^{\frac{1}{m}}\mu.
    $$
\end{theorem}
In the last section, we show that the mixed type inequality combined with the weak convergence theorem can be used to derive bounded solutions of complex Hessian equations with right-hand-side measures dominated by capacity, a result parallel to \cite{KN21}:
\begin{theorem}\label{main Hessian equation introduction}
    Let $\mu$ be a positive Radon measure on $X$ such that $\mu\leq A\operatorname{Cap}_{\gamma,\omega,m}^\tau$ for some $A>0$ and $1<\tau<\frac{n}{n-m}$. Moreover, assume that $\mu$ is absolutely continuous with respect to $H_m(\varphi)=(\gamma+dd^c\varphi)^m\wedge\omega^{n-m}$ for some $\varphi\in\operatorname{SH}_m(X,\gamma,\omega)\cap L^{\infty}(X)$. Then, there exist a unique constant $c>0$ and a function $u\in\operatorname{SH}_m(X,\gamma,\omega)\cap L^{\infty}(X)$ such that
    $$
H_m(u)=c\mu.
    $$
\end{theorem}
To prove \cref{main Hessian equation introduction}, we start from the resolution of the special Hessian equation (\cref{Hessian equation for fHm}) and proceeding approximations time by times to to get the final solution.
\begin{remark}
    In the special case where the metric $\omega$ is locally conformal K\"ahler and $\mu$ is absolutely continuous with respect to the Lebesgue measure on $X$, continuous solutions were obtained by \cite{LeV23}. Our approach primarily employs a monotone approximation and relies heavily on the explicit Cegrell-type decomposition. It is therefore desirable to remove the assumption $\mu<<H_m(\varphi)$ in \cref{main Hessian equation introduction}. Moreover, it is possible to obtain the continuity of the solution in \cref{main Hessian equation introduction} by establishing a stability result, as illustrated in the Monge-Amp\`ere case by \cite{KN21}.
\end{remark}

The paper is organized as follows. In \cref{section 2}, we briefly recall some basic definitions and properties of $(\omega,m)$-subharmonic functions, including the comparsion principle, envelope construction and the domination principle. In \cref{section 3}, we prove \cref{main weak convergence introduction}, using corresponding techniques of Ko\l odziej in the local setting. In \cref{section 4} we prove \cref{local mix type introduction} and we finally give the resolution of \cref{main Hessian equation introduction} in \cref{section 5}.
\subsection*{Acknowledgements}
The author would like to thank Zhiwei Wang (his advisor) and Professor Ngoc Cuong Nguyen for their helpful comments on an earlier version of this manuscript.

\section{preliminaries}\label{section 2}
In this section, we first recall some basic facts about $(\omega,m)$-subharmonic functions. Let $(X,\omega)$ be a compact Hermitian manifold of dimension $n$, we also use the notation ($\Omega,\omega)$ to denote a domain equipped with a Hermitian form $\omega$ in $\mathbb{C}^n$. Fix an integer $m$ such that 1$\leqslant m\leqslant n$. Throughout this paper, we denote by $d=\partial+\bar{\partial}$ to be the usual exterior derivative and $d^c:=\frac{1}{2i}(\partial-\bar{\partial})$ to be a real operator such that $dd^c=i\partial\bar{\partial}$.
\begin{definition}\label{definition: m-positive}
    A smooth real $(1,1)$-form $\alpha$ is called $(\omega,m)$-positive ($m$-positive for short if there is no confusion with $\omega$) in $\Omega$ if the following inequalities hold pointwise in the sense of smooth forms:
    $$
    \alpha^{k}\wedge\omega^{n-k}>0, \quad k=1,...,m.
    $$
    We will use the notation $\Gamma_m(\omega)$ to denote the open convex cone of all $(\omega,m)$-positive $(1,1)$-forms and $\overline{\Gamma_m(\omega)}$ its closure. At any fixed point of $X$, we can diagonalize $\omega$ with respect to $\alpha$ and let $\lambda_1,...,\lambda_n$ be the eigenvalues of $\alpha$ with respect to $\omega$. Then, the above condition reads  
    $$
S_{k,\omega}(\alpha)>0,\quad 1\leq k\leq m,
    $$
    where $S_{k,\omega}(\alpha):=\underset{1\leq j_1<...<j_k\leq n}{\sum}\lambda_{j_1}...\lambda_{j_k}$ is the $k$-th symmetric polynomial of the eigenvalues of $\alpha$ with respect to $\omega$. 

    A function $u\in C^{2}(\Omega)$ is called $(\omega,m)$-subharmonic, denoted $u\in \operatorname{SH}_{m}(\Omega,\omega)$, if $dd^{c}u$ lies in $\overline{\Gamma_m(\omega)}$ at all points of $X$. It is called strictly $(\omega,m)$-subharmonic if $dd^cu$ lies in $\Gamma_m(\omega)$ pointwise on $X$.
\end{definition}

\begin{definition}\label{def: def of m-sh}
    A function $L_{loc}^1(\Omega,\omega^n)\ni u:\Omega\to\mathbb{R}\cup\{-\infty\} $ is called $(\omega,m)$-subharmonic if it is strongly upper semi-continuous and such that
$dd^cu$ is an $(\omega,m)$-positive current, i.e., for arbitrary m-positive (1,1)-forms $\alpha_1,...\alpha_{m-1}$ on $\Omega$, the following inequality holds in the weak sense of currents:
    $$
    dd^{c}u\wedge\alpha_1\wedge...\wedge\alpha_{m-1}\wedge\omega^{n-m}\geq 0.
    $$
Here we say $u$ is strongly upper semi-continuous if $\forall x\in\Omega$, $u(x)=\underset{\Omega\ni y\rightarrow x}{\operatorname{ess}\limsup}\,u(y):=\underset{r\searrow0}{\lim}\operatorname{ess}\underset{B_r(x)}{\sup}u$ .
\end{definition}
\begin{definition}
    Let $(X,\omega)$ be a compact Hermitian manifold and let $\gamma\in\Gamma_m(\omega)$ be an $(\omega,m)$-positive form. A function $u$ on $X$ is called $(\gamma,\omega,m)$-subharmonic, denoted $u\in \operatorname{SH}_m(X,\gamma,\omega)$, if it can be locally written as a sum of a smooth function and an $(\omega,m)$-subharmonic function, and globally for any $(m-1)$-tuple of forms $\alpha_1,...,\alpha_{m-1}$ lie in $\Gamma_m(\omega)$,
    $$
(\gamma+dd^cu)\wedge\alpha_1\wedge...\wedge\alpha_{m-1}\wedge\omega^{n-m}\geq0\quad on\;X.
    $$
\end{definition}

For other equivalent definitions, see \cite{PSWZ25} and \cite[section 9]{GN18}. 
We next review the definition of Hessian measures with respect a background Hermitian form from \cite{KN25b}:
\begin{definition}\label{def: Hessian measures in the local setting}
    Let $\omega$ be a Hermitian metric on a bounded domain $\Omega\subset\mathbb{C}^n$ and $u_1,...,u_m$ be $(\omega,m)$-subharmonic functions on $\Omega$. We can inductively define 
    $$
dd^cu_{p+1}\wedge...\wedge dd^cu_1:=dd^c(u_{p
+1}dd^cu_p\wedge...\wedge dd^cu_1)
    $$
    as closed real currents of order $0$ on $\Omega$. Then
    $$
H_p(u_1,...,u_p):=dd^cu_{p}\wedge...\wedge dd^cu_1\wedge\omega^{n-m},\quad 1\leq p\leq m.
    $$
    is a well-defined closed positive current (positive Radon measure when $p=m$) on $\Omega$. When $u_1=...=u_m=u$, we write
    $$
H_p(u):=(dd^cu)^p\wedge\omega^{n-m},\quad1\leq p\leq m.
    $$
    \end{definition}

   In the sequel of this paper, unless otherwise stated, we will always use the letter $\gamma$ to denote a smooth strictly $(\omega,m)$- positive $(1,1)$-form on $X$. For each $u\in\operatorname{SH}_m(X,\gamma,\omega)\cap L^\infty(X)$, choose a coordinate ball $B\subset X$ and $\phi\in\operatorname{PSH}(B)\cap C^\infty(X)$ such that $dd^c\phi\geq\gamma$, we can define on $B$ the associated Hessian measure of $u$ by
  \begin{align*}
H_{m}(u)&:=(\gamma+dd^cu)^m\wedge\omega^{n-m}=[dd^c(u+\phi)+(\gamma-dd^c\phi)]^m\wedge\omega^{n-m}\\
&=\sum_{k=0}^{m}C_m^k[dd^c(u+\phi)]^k\wedge(\gamma-dd^c\phi)^{m-k}\wedge\omega^{n-m}\\
&=\sum_{k=0}^{m}C_m^kH_k(u+\phi)\wedge(\gamma-dd^c\phi)^{m-k},
    \end{align*}
    This is a well-defined positive Radon measure on $B$ and is clearly independent of the choice of $\phi$. Now, proceeding a partition of unity, we can glue these local measures globally on $X$.

    As illustrated in the main theorem of \cite{Fang25a}, $(\gamma,\omega,m)$-subharmonic functions have nice integrability properties:
\begin{proposition}\label{integrability}
  For each $1<p<\frac{n}{n-m}$, we have $\operatorname{SH}_m(X,\gamma,\omega)\subset L^p(X,\omega^n)$.
\end{proposition}
We record the following basic compactness result:
\begin{proposition}\cite[Lemma 3.3]{KN16} \cite[Proposition 2.15]{PSWZ25} \label{L^1 compactness}
     There is a uniform constant $C=C(X,\gamma,\omega)$,  such that for any $\varphi\in \operatorname{SH}_m(X,\gamma,\omega)$, $\sup_X\varphi=0$, we have    
$$
\int_X|\varphi|\omega^n\leq C.    
$$
In particular, the family $\{\varphi\in \operatorname{SH}_m(X,\gamma,\omega),\;\sup_X\varphi=0\}$ is relatively compact with respect to the $L^1(\omega^n)$-topology in $\operatorname{SH}_m(X,\gamma,\omega)$.
\end{proposition}
    
    We record the domination principles proved in \cite{PSWZ25}, which will be useful later:
    \begin{proposition}\cite[Proposition 5.2]{PSWZ25} \label{domination for beta}
 Fix a constant $0\leq c<1$. Assume $u,v\in\operatorname{SH}_m(X, \gamma,\omega) \cap L^{\infty}(X)$ satisfies $\mathds{1}_{\{u<v\}}H_{m}(u)\leq c\mathds{1}_{\{u<v\}}H_{m}(v)$, then $u\geq v$.
\end{proposition}

\begin{corollary}\cite[Corollary 5.3]{PSWZ25}
\label{domination for beta 2}
   Notations as \cref{domination for beta}, if
    $$
e^{-\lambda u}H_{m}(u)\leq e^{-\lambda v}H_{m}(v),
    $$
    then $u\geq v$.
\end{corollary}

\begin{corollary}\cite[Corollary 5.4]{PSWZ25} \label{domination for beta 3}
Notations as \cref{domination for beta}, if
$$
H_{m}(u)\leq cH_{m}(v)
$$
for some constant $c>0$, then $c\geq 1$.
\end{corollary}
    \begin{definition}\label{def of cap_gamma}
    Let $E$ be a Borel subset of $X$. We define
    $$\operatorname{Cap}_{\gamma,\omega,m}(E):=\sup \left\{\int_{E} H_{m}(u)| u\in\operatorname{SH}_m(X, \gamma,\omega) \cap L^{\infty}(X), -1 \leq u\leq 0\right\},$$
     Note that the set $\{u\in\operatorname{SH}_m(X, \gamma,\omega) \cap L^{\infty}(X), -1 \leq u\leq 0\}$ is non-empty. A Borel set $E\subset X$ is said to be $m$-polar if it is locally contained in the polar set of an $(\omega,m)$- subharmonic function. A measure $\mu$ on $X$ is called non-m-polar if it does not charge any m-polar sets. We refer the reader to the recent paper \cite{KN25c} for some characterizations of $m$-polar sets.
\end{definition}
\begin{theorem}\cite[Theorem 4.9]{KN25b}\cite[Proposition 5.5]{KN25c}
Every $\varphi\in \operatorname{SH}_{m}(X,\gamma,\omega)$ is quasi-continuous with respect to $\operatorname{Cap}_{\gamma,\omega,m}(\cdot)$, i.e., for every $\varepsilon>0$, there exists an open set $U\subset X$ such that $\operatorname{Cap}_{\gamma,\omega,m}(U)<\varepsilon$ and $\varphi$ is continuous on $X-U$. 
\end{theorem}
\begin{definition}
A sequence of Borel functions $f_j$ is said to converge to $f$ in capacity on $X$ if for each fixed $\delta>0$, we have  $$\lim_{j\rightarrow +\infty} \operatorname{Cap}_{\gamma,\omega,m}\left( \left\{|f_j-f|>\delta \right \}\right)=0.$$
\end{definition}

The following convergence lemma established in \cite{PSWZ25} will be crucial in the sequel:
\begin{theorem}\cite[Theorem 3.13]{PSWZ25} \label{thm: weak convergence lemma}
     Let $U \subset \mathbb{C}^n$ be an open set. Suppose $\left\{f_j\right\}_j$ are uniformly bounded quasi-continuous functions which converge in capacity to another quasi-continuous function $f$ on U. Let $\left\{u_1^j\right\}_j,\left\{u_2^j\right\}_j, \ldots,\left\{u_m^j\right\}_j$ be uniformly bounded  $(\omega,m)$- subharmonic  functions on $\Omega$, converging in capacity to bounded $(\omega,m)$- subharmonic functions $ u_1, u_2, \ldots, u_m$ respectively. Then we have the following weak convergence of measures:
$$
f_j  dd^c u_1^j \wedge dd^c u_2^j \wedge \ldots \wedge dd^c u_m^j \wedge \omega^{n-m}\rightarrow f dd^c  u_1 \wedge dd^c u_2 \wedge \ldots \wedge dd^c u_m \wedge \omega^{n-m} .
$$
\end{theorem}
Next, we recall the theory of Perron envelopes for $(\gamma,\omega,m)$- subharmonic functions developed in \cite{PSWZ25}:
\begin{definition}
    Let $h: X \rightarrow \mathbb{R}$ be a bounded measurable function, we set
$$
P_{\gamma,m}(h):=\sup \{u \in \operatorname{SH}_m(X, \gamma,\omega), u \leq h  \quad\text{quasi-everywhere}\}^*,
$$ here "quasi-everywhere" means outside a $m$-polar set.
\end{definition}
The most important properties of envelopes are the following two:
\begin{theorem}\cite[Theorem 2.7]{PSWZ25} \label{contact}
Let $\gamma\in\Gamma_m(\omega)$ be an $(\omega,m)$-positive form. If $h$ is quasi-continuous bounded function, then the complex Hessian measure $H_{m}(P_{\gamma,m}(h))=(\gamma+dd^cP_{\gamma,m}(h))^m\wedge\omega^{n-m}$ is concentrated on the contact set $\mathcal{C}_m=\left\{P_{\gamma, m}(h)=h\right\}$.
\end{theorem}
\begin{corollary}\cite[Corollary 2.2]{PSWZ25} \label{mass concentration 2}
    Let $\gamma$ be as in \cref{contact}, $u,v$ be bounded $(\gamma,\omega,m)$-subharmonic functions and let $w:=P_{\gamma,m}(u,v)=P_{\gamma,m}(\min(u,v))$ be the rooftop envelope of $u,v$. Then
    $$
H_{m}(w)\leq\mathds{1}_{\{w=u\}}H_{m}(u)+\mathds{1}_{\{w=v\}}H_{m}(v).
    $$
\end{corollary}
For further properties of capacities and envelopes, we refer to \cite{KN25c} and \cite{PSWZ25}.
\section{Weak convergence of complex hessian operators}\label{section 3}
In this section, fix a strictly $(\omega,m)$- positive form $\gamma\in\Gamma_m(\omega)$. Motivated by \cite{KN25a}, the main result in this section is as follows:
\begin{theorem}\label{main weak convergence}
    Let $\{u_j\}$ be a sequence of uniformly bounded $(\gamma,\omega,m)$-subharmonic functions. Assume that $H_{m}(u_j)\leq C_1H_{m}(\varphi_j)+C_2H_{m}(\psi_j)$ for some uniformly bounded sequence $\varphi_j\rightarrow\varphi \in\operatorname{SH}_m(X,\gamma,\omega)$ and $\psi_j\rightarrow\psi\in \operatorname{SH}_m(X,\gamma,\omega)$ in capacity. If $u_j\rightarrow u$ in $L^1(X)$, then a subsequence of $u_j$ converges in capacity to $u$. In particular, a subsequence of $H_{m}(u_j)$ converges weakly to $H_{m}(u)$.
\end{theorem}
We first establish a couple of lemmas analogous to those in \cite{KN25a}, the local versions are essentially contained in \cite{KN25b}.
In the sequel, let
\[
P_{\gamma.0} := \left\{v \in \operatorname{SH}_m(X, \gamma,\omega) \cap L^{\infty}(X) \mid \sup_X v = 0 \right\}
\]
be a relatively compact subset in $\operatorname{SH}_m(X, \gamma,\omega)$. Following the proof in \cite{KN25a}, we have the following lemma:
\begin{lemma}\label{lem: convergence}
    Let $\mu$ be a non-m-polar Radon measure on $X$. Suppose $\{u_j\} \subset P_{\gamma,0}$ is a uniformly bounded sequence converging almost everywhere with respect to the Lebesgue measure to $u \in P_{\gamma,0}$. Then there exists a subsequence, still denoted by $\{u_j\}$, such that
    \[
    \lim_{j \to \infty} \int_X |u_j - u| \, d\mu = 0.
    \]
\end{lemma}
\begin{proof}
    The proof is a trivial adaptation of \cite[Lemma 8.3]{KN25b} and \cite[Corollary 8.4]{KN25b}.
\end{proof}
The next lemma is a global version of \cite[Lemma 8.5]{KN25b}, the arguments there are readily applicable, we give a sketch of proof for the convenience of the reader.
\begin{lemma}\label{lem: cap convergence lemma2}
    Let $\{u_j\}$ be as in \cref{main weak convergence}, and let $\{w_j\} \subset P_{\gamma,0}$ be a uniformly bounded sequence that converges in capacity to $w \in P_{\gamma,0}$. Then,
    \[
    \lim_{j \to \infty} \int_X |u - u_j| H_{m}(w_j) = 0.
    \]
\end{lemma}

\begin{proof}
    Observe that
    \[
    |u_j - u| = (\max\{u_j, u\} - u_j) + (\max\{u_j, u\} - u).
    \]
    Let $\phi_j := \max\{u_j, u\}$ and $v_j := \left(\sup_{k \geq j} \phi_k\right)^*$. We have $\phi_j \geq u$ and $\{v_j\}_j$ decreases to $u$ pointwise. By \cite[Proposition 4.3]{KN25b} for any $\delta > 0$,
    \[
    \operatorname{Cap}_{\omega,m}\left(\{|\phi_j - u| > \delta\}\right) = \operatorname{Cap}_{\omega,m}\left(\{\phi_j > u + \delta\}\right) \leq \operatorname{Cap}_{\omega,m}\left(\{v_j > u + \delta\}\right) \to 0,
    \]
    hence $\phi_j\rightarrow u$ in capacity and it follows from \cref{thm: weak convergence lemma} that
    \begin{align*}
       \lim_{j\rightarrow\infty} \int_X (\phi_j - u)H_{m}(w_j)=0.
    \end{align*}
    We next turn to the estimate of the term  $\int_X (\phi_j - u_j)H_{m}(w_j)$. For $j > k$,
    \begin{align*}
        & \int_X (\phi_j - u_j)H_{m}(w_j) - \int_X (\phi_j - u_j) H_{m}(w_k)\\
        &= \int_X (\phi_j - u_j) dd^c (w_j - w_k) \wedge T \\
        &= \int_X (w_j - w_k) dd^c [(\phi_j - u_j) \wedge T] \\
    \end{align*}
    where $T = \sum_{s=0}^{m-1} (\gamma + dd^c w_j)^s \wedge (\gamma + dd^c w_k)^{m-1-s}\wedge\omega^{n-m}$ is a positive current. Set $h_j:=\phi_j-u_j$, we have
    \begin{align*}
        dd^c(h_jT)=dd^ch_j\wedge T+dh_j\wedge d^cT-d^ch_j\wedge dT+h_jdd^cT.
    \end{align*}
    We will deal with the four terms above respectively. For the first term, we can write
    \begin{align*}
\int_X(w_j-w_k)dd^ch_j\wedge T\leq\int_X|w_j-w_k|(\gamma_{w_j}+\gamma_{w_k})\wedge T\rightarrow0
    \end{align*}
    as $j,k\rightarrow\infty$. Where the convergence follows from \cref{thm: weak convergence lemma} and the assumption that $w_j\rightarrow w$ in capacity.

    Since the second term and the third term are mutually conjugate, we will only deal with the second term. We may rewrite the second term as $dh_j\wedge d^cT=dh_j\wedge d^c\gamma\wedge T^\prime+dh_j\wedge d^c\omega\wedge T^{\prime\prime}$, where $T^\prime$ is   a $(n-2,n-2)$-positive current of type $\gamma_{w_1}\wedge...\wedge\gamma_{w_{m-2}}\wedge\omega^{n-m}$ and  $T^{\prime\prime}$ looks like $\gamma_{w_1}\wedge...\wedge\gamma_{w_{m-1}}\wedge\omega^{n-m-1}$, which is a well defined current of order zero thanks to \cite{KN25b}. By the Cauchy-Schwarz inequality \cite[Lemma 2.3]{KN25b} and \cite[Corollary 2.5]{KN25b}, we deduce that
    \begin{align*}
        &\int_X(w_j-w_k)dh_j\wedge d^cT\\
        =&\int_X(w_j-w_k)dh_j\wedge d^c\omega\wedge T^{\prime\prime}+\int_X(w_j-w_k)dh_j\wedge d^c\gamma\wedge T^\prime\\
        \leq&C\int_X|w_j-w_k|dh_j\wedge d^ch_j\wedge\omega\wedge T^\prime\int_X|w_j-w_k|\omega^2\wedge T^\prime\\
        &+C\int_X|w_j-w_k|dh_j\wedge d^ch_j\wedge(\sum_{j=1}^{m-1}\gamma_{w_j})^{m-1}\wedge\omega^{n-m}\int_X|w_j-w_k|(\sum_{j=1}^{m-1}\gamma_{w_j})^{m-1}\wedge\omega^{n-m+1}.
    \end{align*}
   Note that \cite[Lemma 2.3]{KN25b} and \cite[Corollary 2.5]{KN25b} were stated for smooth forms, the general case follows easily by an approximation argument using \cite[Theorem 9.4]{PSWZ25} (note that here we need to use \cite[Proposition 3.20]{KN25b} for convergence of real currents of order zero). We remark that here the constant $C$ in \cite[Lemma 2.3]{KN25b} and \cite[Corollary 2.5]{KN25b} depends only on $n,m,\omega$ but not $\gamma$ (and of course $\gamma_{w_j}$), thus the limiting process does not cause any problems. For example, we know that the inequality
   \begin{align*}
       &\int_X(w_j-w_k)dh_j\wedge d^c\omega\wedge T^{\prime\prime}\\
       \leq&C\int_X|w_j-w_k|dh_j\wedge d^ch_j\wedge(\sum_{j=1}^{m-1}\gamma_{w_j})^{m-1}\wedge\omega^{n-m}\int_X|w_j-w_k|(\sum_{j=1}^{m-1}\gamma_{w_j})^{m-1}\wedge\omega^{n-m+1},
   \end{align*}
  is valid for smooth data $(h_j,w_j)$. Now we assume that each $w_j$ is smooth, since $h_j$ can be written as the difference of two $(\gamma,\omega,m)$-subharmonic functions, by approximating and using \cite[Proposition 3.20]{KN25b} together with \cref{thm: weak convergence lemma}, we get the desired inequality. Finally, when $w_j$ is merely bounded, we may as well proceed an approximation to conclude it.
   
   Since $h_j$ are uniformly bounded and $w_j\rightarrow w$ in capacity, it is easy to see from \cref{thm: weak convergence lemma} that the second term converges to $0$.

    Thanks to \cite[Corollary 2.4]{KN16}, the last term can be handled similarly as the previous terms, as described in \cite[Lemma 8.5]{KN25b}. Overall, we have proved that $ |\int_X (\phi_j - u_j)H_{m}(w_j) - \int_X (\phi_j - u_j) H_{m}(w_k)|\rightarrow0$ as $j,k\rightarrow\infty$. Now for each $\epsilon>0$, fix $k_0$ such that $| \int_X (\phi_j - u_j)H_{m}(w_j) - \int_X (\phi_j - u_j) H_{m}(w_k)|<\epsilon$ for any $j,k\geq k_0$ and finally we can estimate:
    \begin{align*}
        \int_X (\phi_j - u_j)H_{m}(w_j) &\leq \int_X (\phi_j - u_j)H_{m}(w_{k_0}) \\
        & \quad + \left|\int_X (\phi_j -u_j)H_{m}(w_j) - \int_X (\phi_j - u_j)H_{m}(w_{k_0})\right| \\
        &\leq \int_X |u - u_j| H_{m}(w_{k_0}) + \epsilon.
    \end{align*}
    Applying \cref{lem: convergence}, we conclude the proof.
\end{proof}
    To finish the proof of \cref{main weak convergence}, it remains to show the following criterion of convergence in capacity, which simultaneously generalize \cite[Lemma 2.4]{KN25a} and \cite[Proposition 2.5]{KN25b}:

\begin{lemma}\label{characterzation of convergence in cap}
    Let $\{u_j\}_j$ be a uniformly bounded sequence in $\operatorname{SH}_m(X, \gamma,\omega) \cap L^{\infty}(X)$ such that $u_j\rightarrow u\in\operatorname{SH}_m(X, \gamma,\omega) \cap L^{\infty}(X)$ in $L^1(X)$. Then, a subsequence of $\{u_j\}_j$ converges to $u$ in capacity if and only if
    $$
\lim_{j\rightarrow+\infty}\int_X|u-u_j|H_{m}(u_j)=0.
    $$
\end{lemma}
\begin{proof}
    We will use the technique of envelope theory developed by \cite{PSWZ25}, following the idea of \cite[Theorem 3.1]{ALS24}. If $u_j$ converges to $u$ in capacity, then \cref{thm: weak convergence lemma} directly implies that
     $$
\lim_{j\rightarrow+\infty}\int_X|u-u_j|H_{m}(u_j)=0.
    $$
    We now prove the reverse direction. Up to extracting a subsequence, we may assume that 
    $$
\int_X|u-u_j|H_{m}(u_j)<2^{-j}.
    $$
    Set $v_{j,k}:=P_{\gamma,m}(\min(u_j,...,u_k))$ for each $k\geq j$. It is clear that $v_{j,k}$ decreases to $v_j:=P_{\gamma,m}(\inf_{k\geq j}u_k)$ as $k\rightarrow\infty$ for each $j$ and we can further assume $v_j\nearrow v\in\operatorname{SH}_m(X, \gamma,\omega) \cap L^{\infty}(X)$ almost everywhere. It follows easily from \cref{mass concentration 2} that
    $$
\int_X|v_{j,k}-u|H_{m}(v_{j,k})\leq\sum_{l=j}^k\int_X|u_l-u|H_{m}(u_l)<2^{-j+1}
    $$
    Letting $k\rightarrow+\infty$ and using the decreasing convergence theorem \cite[Lemma 5.1]{KN25b} we obtain
    $$
\int_X|v_{j}-u|H_{m}(v_{j})=\lim_{k\rightarrow+\infty}\int_X|v_{j,k}-u|H_{m}(v_{j,k})\leq\sum_{l=j}^{+\infty}\int_X|u_l-u|H_{m}(u_l)<2^{-j+1}
    $$
    Then, the increasing convergence theorem \cite[Lemma 5.4]{KN25b} yields that
    $$
\int_X|v-u|H_{m}(v)=\lim_{j\rightarrow+\infty}\int_X|v_j-u|H_{m}(v_{j})=0.
    $$
    The domination principle (with $c=0$) \cref{domination for beta} tells us that $v\geq u$ and we also have $v\leq u$ by the definition of $v$, hence $v=u$. Since $v_j\nearrow u$ and $v_j\leq u_j$ by definition, it follows from \cite[Corollary 4.11]{KN25b} that $u_j\rightarrow u$ in capacity.
\end{proof}

\section{Hessian equations and Mixed type inequalities}\label{section 4}
As applications of the weak convergence \cref{main weak convergence} we solve some types of degenerate complex Hessian equations and using them to establish the mixed Hessian inequalities on compact Hermitian manifolds.

We first give an alternative proof of bounded solutions in \cite[Theorem 3.14]{KN16}, \cite[Theorem 9.1]{PSWZ25}.
\begin{theorem}\label{Hessian equation for m-positive gamma}
  Let $0\leq f\in L^p(X)$ with $p>\frac{n}{m}$ and $\int_Xf\omega^n>0$. Then there exists a unique constant $c>0$ and a function $u\in \operatorname{SH}_m(X,\gamma,\omega)\cap L^{\infty}(X)$ such that
        $$
H_m(u)=(\gamma+dd^cu)^m\wedge\omega^{n-m}=cf\omega^n.
        $$
\end{theorem}
\begin{proof}
    Choose a sequence of smooth positive functions $0<f_j\in C^{\infty}(X)$ converges in $L^p(\omega^n)$ to $f$. By \cite[Proposition 21]{Szé18}, there exist constants $c_j>0$ and $u_j\in \operatorname{SH}_m(X,\gamma,\omega)\cap C^\infty(X)$ such that
    \begin{equation}
       H_{m}(u_j)=c_jf_j\omega^n,\quad \sup_Xu_j=0.
    \end{equation}
    As illustrated in \cite[Theorem 9.1]{PSWZ25} and \cite[Lemma 3.13]{KN16}, Garding's inequality yields an uniform upper bound $C$ for $c_j$, while the subsolution theorem \cite[Lemma 11.2]{PSWZ25} and the domination principle \cref{domination for beta 3} gives the lower bound of $c_j>0$. We can thus assume that $\lim_{j\rightarrow+\infty}c_j=c>0$. The $L^\infty$- estimate \cite[Theorem 7.7]{PSWZ25} implies that $u_j$ is uniformly bounded. We can furthermore write
    $$
\int_X|u_j-u|H_{m}(u_j)=\int_X|u-u_j|c_jf_j\omega^n\leq C\|u_j-u\|_{L^q(X)}\rightarrow0.
    $$
    Here we note that $u_j$ is uniformly bounded and hence $L^1$ convergence implies $L^q$ convergence, where $\frac{1}{p}+\frac{1}{q}=1$. \cref{characterzation of convergence in cap} yields that a subsequence of $u_j$ converges in capacity to $u$ and hence the weak convergence $H_{m}(u_j)\rightarrow H_{m}(u)$. 
\end{proof}
Almost the same argument as above we can derive bounded solutions of \cite[Theorem 9.3]{PSWZ25} and, in particular, \cite[Lemma 3.19]{KN16}:
\begin{theorem}\label{twist Hessian equation for m-positive gamma}
   Let $0\leq f\in L^p(X)$ with $p>\frac{n}{m}$ and $\int_Xf\omega^n>0$. Then there exists a unique function $u\in \operatorname{SH}_m(X,\omega,\omega)\cap L^{\infty}(X)$ such that
        $$
H_m(u)=(\omega+dd^cu)^m\wedge\omega^{n-m}=e^uf\omega^n.
        $$
\end{theorem}
\begin{proof}
    The argument is exactly the same as \cref{Hessian equation for m-positive gamma} except that we have to use the arguments of the proof in \cite[Theorem 9.3]{PSWZ25} to derive the uniform bound of $u_j$.
\end{proof}

We move on to solve a special type of Hessian equations, which will be crucial in the sequel:
\begin{proposition}\label{Hessian equation for fHm}
  Let $\mu=fH_{m}(\varphi)+gH_{m}(\psi)$ for two nonnegative functions $0\leq f,g\in L^\infty(X)$ and two bounded potentials $\varphi,\psi\in\operatorname{SH}_m(X, \gamma,\omega) \cap L^{\infty}(X)$. Moreover, assume that there exists a positive number $\delta>0$ such that $f\geq\delta>0$. Then, we can find a bounded solution $u\in\operatorname{SH}_m(X, \gamma,\omega) \cap L^{\infty}(X)$ solving the following equation:
    $$
H_{m}(u)=e^{u}\mu.
    $$
\end{proposition}
\begin{proof}
    It follows from \cite[Theorem 9.4]{PSWZ25} that we can find decreasing sequences of smooth $(\gamma,\omega,m)$-subharmonic functions $\varphi_j\searrow\varphi$ and $\psi_j\searrow\psi$. Since $H_{m}(\varphi_j)$ and $H_{m}(\psi_j)$ are smooth volume forms, we can invoke \cref{twist Hessian equation for m-positive gamma} to find $u_j\in\operatorname{SH}_m(X, \gamma,\omega) \cap L^{\infty}(X)$ solving
    $$
H_{m}(u_j)=e^{u_j}(fH_{m}(\varphi_j)+gH_{m}(\psi_j))
    $$
    Since $\varphi_j,\psi_j,f,g$ are all uniformly bounded, we can write
    $$e^{-u_j}H_{m}(u_j)\leq Ce^{-\frac{\varphi_j+\psi_j}{2}}H_{m}(\frac{\varphi_j+\psi_j}{2})=e^{-(\frac{\varphi_j+\psi_j}{2}-\log C)}H_{m}(\frac{\varphi_j+\psi_j}{2}-\log C).$$
    The domination principle \cref{domination for beta 2} yields that $u_j\geq\frac{\varphi_j+\psi_j}{2}-\log C$. Since we have assumed that $f\geq\delta>0$, we can also write
    $$
e^{-u_j}H_{m}(u_j)\geq\delta H_{m}(\varphi_j)\geq C\delta e^{-\varphi_j}H_{m}(\varphi_j)=e^{-(\varphi_j-\log C\delta)}H_{m}(\varphi_j-\log C\delta).
    $$
    This implies that $u_j\leq\varphi_j-\log C\delta$ and hence $u_j$ is uniformly bounded. Assume $u_j\rightarrow u$ in $L^1(X)$, then \cref{main weak convergence} (applied for $\gamma=\gamma=\omega$) gives that $u_j\rightarrow u$ in capacity and hence we conclude by \cref{thm: weak convergence lemma}. 
\end{proof}

We are now in a position to prove the mixed type inequalities on compact Hermitian manifolds, we will mainly follow the ideas in \cite{DL15}.
\begin{lemma}\label{mix type lemma}
    Let $u,v,\varphi,\psi_1,\psi_2\in\operatorname{SH}_m(X,\gamma,\omega)\cap L^{\infty}(X)$ and $0\leq f,g,h_1,h_2\in L^\infty(X)$ be non-negative bounded functions on $X$ such that
    $$
H_{m}(u)=fH_{m}(\varphi)+h_1H_{m}(\psi_1),\quad H_{m}(v)=gH_{m}(\varphi)+h_2H_{m}(\psi_2).
    $$
    Then for each $1\leq k\leq m-1$,
    $$
\gamma_u^k\wedge\gamma_v^{m-k}\wedge\omega^{n-m}\geq f^{\frac{k}{m}}g^{\frac{m-k}{m}}H_{m}(\varphi).
    $$
\end{lemma}
\begin{proof}
    \textbf{Step 1.} We first assume that $\min(f,g)\geq\delta>0$ and that $\varphi,\psi_1,\psi_2$ are smooth strictly $(\gamma,\omega,m)$- subharmonic. Let $f_j,g_j,h_1^j,h_2^j$ be smooth positive approximations of $f,g,h_1,h_2$ in $L^p(X)$ respectively for some $p<\frac{n}{m}$. Applying \cite[Proposition 3.21]{Szé18} we can find $u_j\in\operatorname{SH}_m(X,\gamma,\omega)\cap C^{\infty}(X)$ with $\sup_Xu_j=\sup_Xu$ and $c_j>0$ solving
    $$
H_{m}(u_j)=c_jf_jH_{m}(\varphi)+c_jh_1^jH_{m}(\psi_1).
    $$
    A standard argument using Garding's inequality and the uniform estimates for Hessian equations (see \cite[Theorem 9.1]{PSWZ25}) gives that $c_j\rightarrow c>0$ and that $\{u_j\}_j$ is uniformly bounded. Since $\sup_Xu_j=\sup_Xu$, we can extract a subsequence of $u_j$ such that $u_j\rightarrow\tilde{u}\in\operatorname{SH}_m(X,\gamma,\omega)\cap L^{\infty}(X)$ in $L^1(X)$. \cref{main weak convergence} yields that $u_j\rightarrow\tilde{u}$ in capacity and hence
    $$
H_{m}(\tilde{u})=cfH_{m}(\varphi)+ch_1H_{m}(\psi_1).
    $$
    Combined with the equation for $u$, we must have $c=1$ by the domination principle \cref{domination for beta 3} and hence $\tilde{u}=u$ by the uniqueness of the solution (see \cite[Theorem 15.1]{PSWZ25}), note that we can indeed apply that uniqueness result since we have assumed that $f\geq\delta>0$. Do the same thing for $v$ we will obtain a smooth sequence of functions $v_j$ converges in capacity to $v$ satisfying 
    $$
H_{m}(v_j)=b_jf_jH_{m}(\varphi)+b_jh_2^jH_{m}(\psi_2).
    $$
    with $b_j>0$ and $\lim_{j\rightarrow+\infty}b_j=1$.
    Finally, Garding's inequality implies that
    $$
\gamma_{u_j}^k\wedge\gamma_{v_j}^{m-k}\wedge\omega^{n-m}\geq\min(c_j,b_j)f_j^{\frac{k}{m}}g_j^{\frac{m-k}{m}}H_{m}(\varphi).
    $$
    Letting $j\rightarrow+\infty$ we get the desired inequality.
    
    \textbf{Step 2.} Assume that $\min(f,g)\geq\delta>0$ and that $f,g$ are quasi-continuous (with respect to $\operatorname{Cap}_{\gamma,\omega,m}$) on $X$. As in \cref{Hessian equation for fHm}, we choose sequences of smooth and strictly $(\gamma,\omega,m)$- subharmonic functions $\varphi_j,\psi_1^j,\psi_2^j$  decreasing to $\varphi,\psi_1,\psi_2$ respectively. We may then use \cref{twist Hessian equation for m-positive gamma} to find $u_j\in\operatorname{SH}_m(X,\gamma,\omega)\cap L^{\infty}(X)$ solving
    $$
H_{m}(u_j)=e^{u_j-u}[fH_{m}(\varphi_j)+h_1H_{m}(\psi_1^j)].
    $$
    Since $\varphi_j$ is uniformly bounded and $\delta\leq f\leq C$, exactly the same argument as in \cref{Hessian equation for fHm} shows that $u_j$ is uniformly bounded and hence converges in capacity to a function $u_\infty\in\operatorname{SH}_m(X,\gamma,\omega)\cap L^{\infty}(X)$ by \cref{main weak convergence}. \cref{thm: weak convergence lemma} yields that 
    $$
e^{-u_\infty}H_{m}(u_\infty)=e^{-u}H_{m}(u).
    $$
    Hence $u_\infty=u$ by the domination principle \cref{domination for beta 2}. Do the same thing for $v$ we obtain a sequence $v_j$ converges in capacity to $v$. Applying Step 1 for $u_j,v_j$ we get
$$
\gamma_{u_j}^k\wedge\gamma_{v_j}^{m-k}\wedge\omega^{n-m}\geq e^{\frac{k(u_j-u)}{m}}e^{\frac{(m-k)(v_j-v)}{m}}f^{\frac{k}{m}}g^{\frac{m-k}{m}}H_{m}(\varphi_j).
$$
Letting $j\rightarrow\infty$ we conclude the proof of this step by using \cref{thm: weak convergence lemma}.

\textbf{Step 3.} Assume only $\min(f,g)\geq\delta>0$. Choosing two continuous uniformly bounded sequences $f_j,g_j$ such that $\min(f_j,g_j)\geq\delta$ and that $f_j,g_j$ converge in $L^1(X,H_{m}(\varphi))$ to $f,g$ respectively. Thanks to \cref{Hessian equation for fHm}, we can find $u_j$ solving
$$
H_{m}(u_j)=e^{u_j-u}[f_jH_{m}(\varphi)+h_1H_{m}(\psi_1)].
$$
The same argument as in Step 2 gives that $u_j\rightarrow u$ in capacity and the result in this step follows easily from Step 2.

\textbf{Step 4.} In general, set $f_j:=\max(f,\frac{1}{j})$ and $g_j:=\max(g,\frac{1}{j})$ and use \cref{Hessian equation for fHm} to solve
$$
H_{m}(u_j)=e^{u_j-u}[f_jH_{m}(\varphi)+h_1H_{m}(\psi_1)].
$$
Then 
$$
e^{-u_j}H_{m}(u_j)\geq e^{-u}[fH_{m}(\varphi)+h_1H_{m}(\psi_1)]=e^{-u}H_{m}(u)
$$
and hence $u_j\leq u$ by \cref{domination for beta 2}. On the other hand, the same argument as in \cref{Hessian equation for fHm} shows that $u_j\geq\frac{\varphi+\psi_1}{2}-C$ and hence $u_j$ is uniformly bounded. We do the same thing for $v$ and applying \cref{main weak convergence} to get $u_j\rightarrow u$ and $v_j\rightarrow v$ in capacity. The result follows from Step 3 and a limiting process easily.
\end{proof}

\begin{theorem}\label{global mix type}
    Let $\mu$ be a positive Radon measure on $X$ such that $\mu$ is absolutely continuous with respect to the Hessian measure $H_m(\varphi)$ for some $\varphi\in\operatorname{SH}_m(X,\gamma,\omega)\cap L^{\infty}(X)$. Let $u,v\in\operatorname{SH}_m(X,\gamma,\omega)\cap L^{\infty}(X)$ be such that
    $$
H_m(u)\geq f\mu,\quad H_m(v)\geq g\mu,
    $$
    where $0\leq f,g\in L^1(\mu)$. Then, for each $1\leq k\leq m-1$, we have
    $$
\gamma_u^k\wedge\gamma_v^{m-k}\wedge\omega^{n-m}\geq f^{\frac{k}{m}}g^{\frac{m-k}{m}}\mu.
    $$
\end{theorem}
\begin{proof}
    By the Radon-Nikodym theorem we can write $\mu=hH_m(\varphi)$ for some $h\in L^1(H_m(\varphi))$. We can thus assume without loss of generality that $\mu=H_m(\varphi)$ for some $\varphi\in\operatorname{SH}_m(X,\gamma,\omega)\cap L^{\infty}(X)$. We first treat the case where $f,g$ are bounded. Fix $\delta>0$, for each $j\geq1$, we use \cref{Hessian equation for fHm} to solve
    $$
H_m(u_j)=e^{j(u_j-u)}[\delta H_m(u)+(1-\delta)f\mu].
    $$
    The domination principle \cref{domination for beta 2} easily gives that $u\leq u_j\leq u-j^{-1}\log\delta$ and hence $u_j$ converges uniformly to $u$ on $X$. Do the same thing for $v$ to get a sequence $v_j$ converges uniformly to $v$. Applying \cref{mix type lemma} we can write
    $$
\gamma_{u_j}^k\wedge\gamma_{v_j}^{m-k}\wedge\omega^{n-m}\geq e^{\frac{kj(u_j-u)}{m}}e^{\frac{(m-k)j(v_j-v)}{m}}(1-\delta)f^{\frac{k}{m}}g^{\frac{m-k}{m}}\mu\geq(1-\delta)f^{\frac{k}{m}}g^{\frac{m-k}{m}}\mu,
    $$
  where the second inequality is because $u_j\geq u$ and $v_j\geq v$. Letting $j\rightarrow\infty$ we obtain
    $$
\gamma_{u}^k\wedge\gamma_{v}^{m-k}\wedge\omega^{n-m}\geq (1-\delta)f^{\frac{k}{m}}g^{\frac{m-k}{m}}\mu.
    $$
    We can then let $\delta\rightarrow0$ to conclude the proof.

    Finally, when $f,g\in L^1(H_m(\varphi))$, setting $f_j:=\min(f,j)$ and $g_j:=\min(g,j)$. Applying the last step and letting $j\rightarrow\infty$ we get the desired result.
\end{proof}
As in \cite{DL15}, we illustrate that the local version of mixed Hessian inequalities can be derived from its global version. We first give an extension lemma of local plurisubharmonic functions, which is probably well-known:

\begin{lemma}\label{extension to global}
    Let $B$ be an open subset in $X$ which is biholomorphic to a ball in $\mathbb{C}^n$. Given a function $\rho\in \operatorname{PSH}(B)\cap L^{\infty}_{loc}(B\backslash\{0\})$, there is a large constant $A>0$, a smaller ball $B_1\subset \subset B$ and a function $\hat{\rho}\in \operatorname{PSH}(X,A\omega)\cap L^{\infty}_{loc}(X\backslash\{0\})$ such that $\hat{\rho}|_{B_1}=\rho|_{B_1}$.
\end{lemma}
\begin{proof}
    By shrinking $B$ and adding a constant, we can assume without loss of generality that $\rho<-c<0$ in $B$, where $c>0$ is a positive constant. Choose a smooth plurisubharmonic exhaustion function $\tau$ of the ball $B$ (which is defined in a neighborhood of $\overline{B}$), that is, $B=\{\tau<0\}$. We first choose a ball $B_1\subset\subset B$ and a large constant $C>0$ such that $u\geq C\tau$ on $\partial B_1$ and then a ball $B_2$ such that $B_1\subset\subset B_2\subset\subset B$ and $\rho\leq C\tau$ on $\partial B_2$. This is possible since we have assumed that $u<-c<0$ in $B$. We then let 
$$
\hat{\rho}(z)=
\begin{cases}
\rho(z),  & z\in B_1 \\
\max(u(z),C\tau(z)),    & z\in B_2\backslash B_1 \\
C\tau(z), & z\in B\backslash B_2
\end{cases}.
$$
We have thus get a plurisubharmonic function $\hat{\rho}\in \operatorname{PSH}(B)$ such that $\hat{\rho}|_{B_1}=\rho$ and $\hat{\rho}$ is smooth in $B\backslash B_2$.

Next, we choose a cutoff function $\chi\in C^{\infty}(X)$ such that $\chi\equiv1$ in a neighborhood of $\overline{B_2}$ and $\operatorname{Supp}\chi\subset B$. It follows easily that there is a large constant $A>0$ such that $\chi\cdot u$ satisfies our conditions.

\end{proof}

\begin{theorem}\label{local mix type}
    Let $(B,\omega)$ be a small ball in $\mathbb{C}^n$ equipped with a Hermitian metric $\omega$ and let $\chi_1,...,\chi_m$ be smooth $(1,1)$- forms in $B$. Let $u_j$ be bounded $(\chi_j,\omega,m)$-subharmonic functions such that $(\chi_j+dd^cu_j)^m\wedge\omega^{n-m}\geq f_j\mu$, $1\leq j\leq m$. Where $0\leq f_j\in L^1(\mu)$ and $\mu$ is a positive Radon measure absolutely continuous with respect to the Hessian measure $(dd^c\varphi)^m\wedge\omega^{n-m}$ for some $\varphi\in\operatorname{SH}_m(B,\omega)\cap L^\infty(B)$. Then
    $$
(\chi_1+dd^cu_1)\wedge...\wedge (\chi_m+dd^cu_m)\wedge\omega^{n-m}\geq(f_1...f_m)^{\frac{1}{m}}\mu.
    $$
\end{theorem}
\begin{proof}
 We may assume without loss of generality that $\chi_1=\chi_2=...=\chi_m=0$, the proof of the general case proceeds exactly the same. View $B$ as an open submanifold in the projective space $M:=\mathbb{CP}^n$ and let $g$ be the standard Fubini-Study metric on $M$. We also choose a smooth plurisubharmonic defining function $\rho$ of $B$ and write $\omega_1:=dd^c\rho$, which is a K\"ahler form near $\overline{B}$. Since the problem is local, we may work on a smaller ball $B_1\subset\subset B$. Select a cutoff function $\chi\in C^\infty(X)$ such that $0\leq\chi\leq1$, $\chi|_{\overline{B_1}}\equiv1$ and supp$\chi\subset\subset B$. Set $\omega_1^\prime:=\chi\omega_1+(1-\chi)g$, which is clearly a Hermitian metric on $M$ such that $\omega_1^\prime|_{B_1}=\omega_1$. Similarly, one can construct a Hermitian metric $\omega^\prime$ on $M$ such that $\omega^\prime|_{B_1}=\omega$. 

  Using \cref{extension to global}, one can easily paste a $\log$ function with $\rho$ to construct a quasi-plurisubharmonic function $v\in\mbox{PSH}(X,\omega_1^\prime)\cap L^\infty_{loc}(X-\{0\})\subset\operatorname{SH}_m(X,\omega_1^\prime,\omega^\prime)$ with an isolated singularity at $0$. Here $0\in B_1\subset B$ is the origin. Up to adding a constant, we may assume that $v$ is larger than $u_j-\rho,\varphi-\rho$ for each $j$ near $\partial B_1$. Set now
  $$
u_j^\prime(z)=
\begin{cases}
\max(u_j(z)-\rho(z),v(z)),  & z\in B_1 \\
v(z),    & z\in X\backslash B_1 \\
\end{cases}.
$$  
It is clear that $u_j^\prime\in\operatorname{SH}_m(X,\omega_1^\prime,\omega^\prime)\cap L^{\infty}(X)$ and $u_j^\prime$ coincides with $u_j-\rho$ in $B_1$ (up to shrinking $B_1$ slightly). Similarly, we also glue $\varphi-\rho$ with $v$ to get $\varphi^\prime\in\operatorname{SH}_m(X,\omega_1^\prime,\omega^\prime)\cap L^{\infty}(X)$ such that $\varphi-\rho=\varphi^\prime$ in $B_1$. Shrinking $B_1$ if necessary and extending $\mu$ and $f_j^\prime$ by zero, we get a Radon measure $\mu^\prime$ on $X$, which is absolutely continuous with respect to $H_m(\varphi^\prime)=(\omega_1^\prime+dd^c\varphi^\prime)^m\wedge(\omega^{\prime})^{n-m}$ and satisfying $\mu^\prime|_{B_1}=\mu$. The global mixed type inequality \cref{global mix type} yields that
  $$
(\omega_1^\prime+dd^cu_1^\prime)\wedge...\wedge(\omega_1^\prime+dd^cu_m^\prime)\wedge(\omega^{\prime})^{n-m}\geq (f_1^\prime...f_m^\prime)^{\frac{1}{m}}\mu^\prime
  $$
  on $X$. Note that $\omega_1^\prime=\omega_1=dd^c\rho$ and $u_j^\prime-\rho=u_j$ on $B_1$, we conclude the proof by restricting the above inequality to $B_1$.
  
  \end{proof}

\section{Solving Hessian equations for measures dominated by capacity}\label{section 5}
As an application of mixed type inequalities, we solve complex Hessian equations with the right-hand side well dominated by capacities and moreover absolutely continuous with respect to some Hessian measures, establishing a result analogous to \cite{KN21}.

\begin{theorem}\label{main Hessian equation}
    Let $\mu$ be a positive Radon measure on $X$ such that $\mu\leq A\operatorname{Cap}_{\gamma,\omega,m}^\tau$ for some $A>0$ and $1<\tau<\frac{n}{n-m}$, assume moreover that $\mu$ is absolutely continuous with respect to $H_m(\varphi)$ for some $\varphi\in\operatorname{SH}_m(X,\gamma,\omega)\cap L^{\infty}(X)$. Then, there is a constant $c>0$ and a function $u\in\operatorname{SH}_m(X,\gamma,\omega)\cap L^{\infty}(X)$ such that
    $$
H_m(u)=c\mu.
    $$
\end{theorem}

Before going to prove \cref{main Hessian equation}, we first make some preparations. The following version of Chern-Levine-Nirenberg inequalities established in \cite{KN25c} will be needed (in \cite[Theorem 4.1]{KN25c} it was stated for $\gamma$ positive, but their argument is certainly valid when $\gamma$ is just $m$-positive):
\begin{lemma} \cite[Theorem 4.1]{KN25c} \label{CLN ieq}
    Let $\varphi,\psi\in \operatorname{SH}_m(X,\gamma,\omega)\cap L^{\infty} (X)$ be such that $-1\leq\varphi\leq0$ and $\sup_X\psi=0$. Then, there exists a uniform constant $C=C(m,n,\gamma,\omega)$ such that for each $1\leq k\leq m$,
    $$
\int_X(-\psi)\gamma_\varphi^k\wedge\omega^{n-k}\leq C.
    $$
\end{lemma}

As a consequence of the above version of Chern-Levine-Nirenberg inequality, we have the following sharp estimates for the capacity of sublevel sets, due to \cite{KN25c}: 
\begin{corollary} \cite[Theorem 1.1]{KN25c} \label{decay of capacity}
    There is a uniform constant $C$ depending only on $m,n,\gamma,\omega$ such that for any $\psi\in\operatorname{SH}_m(X,\gamma,\omega)$ and $\sup_X\psi=-1$, we have for every $t>0$,
    $$
\operatorname{Cap}_{\gamma,\omega,m}(\psi<-t)\leq\frac{C}{t}.
    $$
\end{corollary}
\begin{proof}
    Fix $\varphi\in\operatorname{SH}_m(X,\gamma,\omega)$ such that $-1\leq\varphi\leq0$. We can write
    \begin{align*}
        t\int_{\{\psi<-t\}}\gamma_{\varphi}^m\wedge\omega^{n-m}\leq\int_{\{\psi<-t\}}(-\psi)\gamma_{\varphi}^m\wedge\omega^{n-m}\leq C,
    \end{align*}
    where the last inequality follows from \cref{CLN ieq}. Taking supreme with respect to $\varphi$ we get the desired estimate.
\end{proof}

We recall the following lemma established in \cite{KN16} and \cite{PSWZ25}:

\begin{lemma}\cite[Lemma 3.8]{KN16} \cite[Corollary 7.4]{PSWZ25} \label{capacity level set lemma}
      Fix $0<\epsilon<\frac{3}{4}$. Consider $u,v\in \operatorname{SH}_m(X,\gamma,\omega)\cap L^{\infty} (X)$ with $u\leq0$ and $-1\leq v\leq0$. Set $S(\epsilon):=\inf_X[u-(1-\epsilon)v]$ and $U(\epsilon,s):=\{u<(1-\epsilon)v+S(\epsilon)+s\}$. Then, there exists $C_1=C_1(m,n,\gamma,\omega)$ such that for any $0<s,t<\frac{(\lambda\epsilon)^3}{10C_1}$, we have
    $$
t^m\operatorname{Cap}_{\gamma,\omega,m}(U(\epsilon,s))\leq C\int_{U(\epsilon,s+t)}(\gamma+dd^cu)^m\wedge\omega^{n-m}.
    $$
    Where $C>0$ is a uniform constant depending only on $m$. 
\end{lemma}
We can now proceed to state the following version of $L^\infty$- estimate:
\begin{theorem}\label{hessian a priori estimate}
     Fix $0<\epsilon<\frac{3}{4}$. Consider $u,v\in \operatorname{SH}_m(X,\gamma,\omega)\cap L^{\infty}(X)$ with $u\leq0$ and $-1\leq v\leq0$. Assume furthermore that
    $$
(\gamma+dd^cu)^m\wedge\omega^{n-m}\leq\mu,\quad \sup_Xu=0
    $$
for some positive Radon measure $\mu$ satisfying $\mu\leq A\operatorname{Cap}_{\gamma,\omega,m}^\tau$, where $A>0,1<\tau<\frac{n}{n-m}$ are constants. Put $U(\epsilon,s):=\{u<(1-\epsilon)v+S(\epsilon)+s\}$ for $s>0$. Then, for any $0<s<\frac{(\lambda\epsilon)^3}{10C_1}$ we have
    $$
s\leq  C\cdot \operatorname{Cap}_{\gamma,\omega,m}(U(\epsilon,s))^{\frac{\tau-1}{m}}.
    $$
    Where $C$ is a constant depending on $n,m,\gamma,X,\omega$. In particular, we have a uniform estimate
    $$
|\inf_Xu|\leq C.
    $$
\end{theorem}
\begin{proof}
    Using \cref{capacity level set lemma} we can write for each $0<s,t<\frac{(\lambda\epsilon)^3}{10C_1}$
     \begin{align*}
         t^m\operatorname{Cap}_{\gamma,\omega,m}(U(\epsilon,s))&\leq C\int_{U(\epsilon,s+t)}(\gamma+dd^cu)^m\wedge\omega^{n-m}\leq C\int_{U(\epsilon,s+t)}\mu\\
         &\leq CA\cdot \operatorname{Cap}_{\gamma,\omega,m}(U(\epsilon,s+t))^\tau.
     \end{align*}
Set $a(s):=\operatorname{Cap}_{\gamma,\omega,m}(U(\epsilon,s))^{\frac{1}{m}}$, then we have
$$
ta(s)\leq Ca(s+t)^{\tau}.
$$
Now, exactly the same argument as \cite[Theorem 3.10]{KN16} yields that
$$
s\leq Ca(s)^{\tau-1}.
$$
Taking $v=0$ and using \cref{decay of capacity} we can further write
\begin{align*}
    s&\leq C\cdot \operatorname{Cap}_{\gamma,\omega,m}(U(\epsilon,s))^{\frac{\tau-1}{m}}=C\cdot \operatorname{Cap}_{\gamma,\omega,m}(u<\inf_Xu+s)^{\frac{\tau-1}{m}}\\
    &\leq \frac{C}{|\inf_Xu+s|^{\frac{\tau-1}{m}}}.
\end{align*}
From the above inequality we get easily the uniform estimate of $u$.
\end{proof}

We will also need the following lemma in the proof of \cref{main Hessian equation}.
\begin{lemma}\label{lower bound of c}
    Let $\mu$ be a positive Radon measure $\mu$ satisfying $\mu\leq A\operatorname{Cap}_{\gamma,\omega,m}^\tau$, where $A>0,1<\tau<\frac{n}{n-m}$ are constants. Then there exists a uniform constant $c_0>0$ depending on $A,\gamma,\omega,m,n,\tau$ such that for any $u\in \operatorname{SH}_m(X,\gamma,\omega)\cap L^{\infty} (X)$ and $c>0$ solving
    $$
H_m(u)=c\mu,
    $$
    we have $c\geq c_0$.
\end{lemma}
\begin{proof}
    Write $S:=\inf_Xu$ and $U(\epsilon,s):=\{u<S+s\}$, \cref{capacity level set lemma} implies that
    $$
t^m\operatorname{Cap}_{\gamma,\omega,m}(U(\epsilon,s))\leq C\int_{U(\epsilon,s+t)}H_m(u)=Cc\mu(X)\leq C^\prime\cdot c,
    $$
 for each $0<s<\frac{(\lambda\epsilon)^3}{10C_1}$. It suffices to establish a lower bound for $\operatorname{Cap}_{\gamma,\omega,m}(U(\epsilon,s))$, but this follows immediately from \cref{hessian a priori estimate}.
\end{proof}
We now return to the proof of \cref{main Hessian equation}.

\begin{proof}[Proof of \cref{main Hessian equation}]
   Since $\mu$ is absolutely continuous with respect to $H_m(\varphi)$, the Radon-Nikodym theorem implies that there is $0\leq f\in L^1(H_m(\varphi))$ such that $\mu=fH_m(\varphi)$.
   
    \textbf{Step 1.} We first assume that $f$ is bounded and that there exists $\delta>0$ such that $f\geq\delta>0$. In this case, we can invoke \cref{Hessian equation for fHm} to solve
    $$
H_m(u_j)=e^{\frac{u_j}{j}}\mu,
    $$
    for each $j\geq1$ and $u_j\in\operatorname{SH}_m(X,\gamma,\omega)\cap L^{\infty} (X)$. Set $v_j:=u_j-\sup_Xu_j$, then the above equation can be rewritten as 
    \begin{equation}\label{eq v_j}
        H_m(v_j)=c_je^{\frac{v_j}{j}}\mu
    \end{equation}
    where $c_j=e^{\frac{\sup_Xu_j}{j}}$. We first claim that $c_j$ is uniformly bounded away from zero. Indeed, since $f\geq\delta>0$, we can write
    \begin{align*}
        e^{-\frac{u_j}{j}}H_m(u_j)\geq\delta H_m(\varphi)\geq\delta_0e^{-{\frac{\varphi}{j}}}H_m(\varphi)=e^{-\frac{\varphi-j\log\delta_0}{j}}H_m(\varphi-j\log\delta_0).
    \end{align*}
    The domination principle \cref{domination for beta 2} yields that $u_j\leq\varphi-j\log\delta_0$ and hence $\frac{\sup_Xu_j}{j}$ is uniformly bounded from above. Similarly one can show that $\frac{u_j}{j}$ is bounded below. Consequently, we can assume $c_j\rightarrow c>0$ up to extracting a subsequence.

    Since $\sup_Xv_j=0$ and $c_j\rightarrow c$, we have $c_je^{\frac{v_j}{j}}\mu\leq A^\prime \operatorname{Cap}_{\gamma,\omega,m}^\tau$ for some uniform constant $A^\prime$, \cref{hessian a priori estimate} yields that $\{v_j\}_j$ is uniformly bounded. Let $v_j\rightarrow v$ in $L^1(X)$, we claim that $e^{\frac{v_j}{j}}\rightarrow1$ in capacity. In fact, we have
    \begin{align*}
        \operatorname{Cap}_{\gamma,\omega,m}(e^{\frac{v_j}{j}}<1-\epsilon)= \operatorname{Cap}_{\gamma,\omega,m}(v_j<j\log(1-\epsilon))\leq\frac{C}{-j\log(1-\epsilon)}\rightarrow0.
    \end{align*}
    Where in the last inequality we have used \cref{decay of capacity}. Recall that in this case $\mu=fH_m(\varphi)$ with $f$ bounded, it follows from \eqref{eq v_j} and \cref{main weak convergence} that $v_j\rightarrow v$ in capacity. Taking the limit in both sides of \eqref{eq v_j} we arrive at
    $$
H_m(v)=c\mu.
    $$

    \textbf{Step 2.} In this step we assume only $f\geq\delta>0$. Set $f_j:=\min(f,j)$, it follows from \textbf{Step 1} that there exist $u_j\in\operatorname{SH}_m(X,\gamma,\omega)\cap L^{\infty} (X)$ and $c_j>0$ such that
    \begin{equation}
        H_m(u_j)=c_jf_jH_m(\varphi),\quad\sup_Xu_j=0.
    \end{equation}
    Since $f_j$ increases to $f$, the domination principle \cref{domination for beta 3} yields that $c_j$ is decreasing. Moreover, since $f_jH_m(\varphi)\leq fH_m(\varphi)=\mu\leq A\operatorname{Cap}_{\gamma,\omega,m}^\tau$, \cref{lower bound of c} implies that $c_j$ is uniformly bounded below and away from zero. Consequently, we can write $c_j\searrow c>0$. Again using \cref{hessian a priori estimate} it is easy to see that $u_j$ is uniformly bounded. Moreover, if we set $u_j\rightarrow u$ in $L^1(X)$, then
    \begin{align*}
        \int_X|u-u_j|H_m(u_j)\leq c_1\int_X|u-u_j|fH_m(\varphi)=c_1\int_X|u-u_j|d\mu\rightarrow0,
    \end{align*}
    where in the last convergence we have used \cref{lem: convergence}. It follows from \cref{characterzation of convergence in cap} that a subsequence of $u_j$ converges in capacity to $u$, whence our desired equation
    $$
H_m(u)=c\mu.
    $$
    \textbf{Step 3.} We finally remove the assumption $f\geq\delta>0$. Set $g_j:=\max(f,\frac{1}{j})$. We can then apply the previous step to obtain $u_j\in\operatorname{SH}_m(X,\gamma,\omega)\cap L^{\infty} (X)$ and $b_j>0$ such that
    \begin{equation}
        H_m(u_j)=b_jg_jH_m(\varphi),\quad\sup_Xu_j=0.
    \end{equation}
    Since $g_j\searrow g$, the domination principle \cref{domination for beta 3} yields that $b_j$ is increasing. We claim that $b_j$ is uniformly bounded above. Indeed, using the mix type inequality \cref{global mix type} we can write
    $$
\gamma_{u_j}\wedge\gamma_{\varphi}^{m-1}\wedge\omega^{n-m}\geq(b_jg_j)^{\frac{1}{m}}H_m(\varphi).
    $$
    Integrating both sides we have that 
$$
b_j^{\frac{1}{m}}\int_Xg_j^{\frac{1}{m}}H_m(\varphi)\leq\int_X\gamma\wedge\gamma_\varphi^{m-1}\wedge\omega^{n-m}+\int_Xu_jdd^c(\gamma_\varphi^{m-1}\wedge\omega^{n-m}).
$$
Since $g_j\searrow f$ and  $\int_XfH_m(\varphi)>0$, we have that $\int_Xg_j^{\frac{1}{m}}H_m(\varphi)$ is bounded below. For the right-hand-side, an easy computation taking into account \cite[Lemma 2.4]{KN16} shows that there is a uniform constant $C$ such that
$$
\int_Xu_jdd^c(\gamma_\varphi^{m-1}\wedge\omega^{n-m})\leq C\sum_{j=1}^3\int_Xu_j\gamma_{\varphi}^{m-j}\wedge\omega^{n-m+j}\leq C_1.
$$
Where the last inequality is due to \cref{CLN ieq}. We have therefore obtained an upper bound for $b_j$. The a priori estimates \cref{hessian a priori estimate} yields the uniform boundedness of $u_j$ and hence we can proceed the same arguments as in Step 2 to finish the proof.
\end{proof}

\begin{remark}
    The solution in \cref{main Hessian equation} is likely to be continuous, as the Monge-Amp\`ere case in \cite{KN21}, we will not pursue it now since our main tool is the convergence theorem \cref{main weak convergence} and the mixed type inequality \cref{global mix type}. It will be very desirable to remove the assumption "$\mu<<H_m(\varphi)$" and jus to assume that $\mu$ is non-$m$-polar in \cref{main Hessian equation}. In the Monge-Amp\`ere setting \cite{KN21}, this is achieved by approximating $\mu$ with regular measures that maintain the capacity domination condition. However, in the Hessian setting, standard local regularizations (e.g., convolution) do not easily preserve the uniform capacity bound $\mu\leq A\operatorname{Cap}_{\gamma,\omega
    ,m}^\tau$ due to the presence of torsion terms and the lack of global potentials. Establishing such an approximation result or a stability estimate for general non-$m$-polar measures remains an open problem. 
\end{remark}

\end{document}